\documentclass{amsart}
\title{Continuous-trace $k$-graph $C^*$-algebras}
\author{Danny Crytser}
\usepackage{enumerate}
\usepackage{hyperref} 
\usepackage{mathtools}
\usepackage{graphicx,amssymb,amsmath,amsthm,amsfonts}

\usepackage{tikz}
\usetikzlibrary{matrix}
\usetikzlibrary{decorations.markings}
\usetikzlibrary{arrows}
\usepackage{hyperref}
\usepackage{amssymb,amsmath,amsthm}

\newcommand{\C}{\mathbb{C}}

\newcommand{\N}{\mathbb{N}}

\newcommand{\Z}{\mathbb{Z}}

\theoremstyle{plain}
\newtheorem{theorem}{Theorem}[section]

\newtheorem{lemma}[theorem]{Lemma}
\newtheorem{corollary}[theorem]{Corollary}
\theoremstyle{definition}
\newtheorem{definition}[theorem]{Definition}

\newtheorem{example}[theorem]{Example}
\theoremstyle{remark}
\newtheorem*{remark}{Remark}

\begin{document}



\begin{abstract}
A characterization is given for directed graphs that yield graph $C^*$-algebras with continuous trace. This is established for row-finite graphs with no sources first using a groupoid approach, and extended to the general case via the Drinen-Tomforde desingularization. Partial results are given to characterize higher-rank graphs that yield $C^*$-algebras with continuous trace.
\end{abstract}

\maketitle

\section{Introduction}

To any directed graph $E$ one can affiliate a graph $C^*$-algebra $C^*(E)$, generated by a universal family of projections and partial isometries that satisfy certain Cuntz-Krieger relations. Many $C^*$-algebraic properties of $C^*(E)$ are governed by graph-theoretic properties of $E$. For example, $C^*(E)$ is an AF algebra exactly when $E$ has no directed cycles. To any graph $E$ there is also affiliated a \emph{path groupoid} $G_E$, an \'{e}tale groupoid which models the shift dynamics of infinite paths in $E$. This groupoid provides an alternate model for $C^*(E)$, in the sense that $C^*(E) \cong C^*(G_E)$. This isomorphism allows for the use of tools from the theory of groupoid $C^*$-algebras to study graph algebras. In this paper we give an example of this approach, characterizing continuous-trace graph $C^*$-algebras by applying the main result of \cite{MRW3} to the path groupoid of a directed graph. 

The path groupoid is easiest to use if the graph is \emph{non-singular}, in the sense that each vertex is the range of a finite non-empty set of edges. We first work in the non-singular case, and then use desingularization to extend to the general case. Desingularization takes a non-regular graph $E$ and returns a non-singular graph $\tilde{E}$ such that the affiliated graph $C^*$-algebras are Morita equivalent (so that continuity of trace is preserved). As an application we use a result from \cite{Tyler} to characterize continuous trace AF algebras in terms of their Bratteli diagrams. 

In the last section we consider higher-rank graph $C^*$-algebras with continuous trace. Higher-rank graphs are categories which generalize the category of finite directed paths within a directed graph. They have $C^*$-algebras defined along the same lines as graph $C^*$-algebras. We include necessary background on the theory of higher-rank graph algebras. Again, the use of groupoids is crucial. The higher-rank case is more complicated and we are only able to give partial results. In particular, giving a combinatorial description of when the isotropy groups vary continuously for a $k$-graph path groupoid seems out of reach, so we focus instead on the principal/aperiodic case. We note a simple necessary condition on a higher-rank graph for its associated $C^*$-algebra to have continuous trace, a corollary of a result from \cite{ES}.

While this paper was in preparation we were made aware of a related paper of Hazlewood (\cite{Hazlewood}) which contains similar results to ours. In particular, Theorems 6.5.22, 6.2.13 and 6.4.11 in \cite{Hazlewood} are in some sense cycle-free/principal versions of Theorem \ref{main-guy} and Theorem \ref{k-big}. The results in \cite{Hazlewood} also show that some desingularization \'{a} la Section 4 in the present paper is possible for the $k$-graph case (although it seems that resolving infinite receivers is somewhat more difficult). 

\section{Continuous-trace $C^*$-algebras; graph algebras; groupoids}

For an element $a$ in a $C^*$-algebra $A$, and a unitary equivalence class $s = [\pi] \in \hat{A}$, we define the rank of $s(a)$ to be the rank of $\pi(a)$. We say that $s(a)$ is a projection if and only if $\pi(a)$ is a projection.

\begin{definition}{\cite[Def. 5.13]{RaeWil}}
Let $A$ be a $C^*$-algebra with Hausdorff spectrum $\hat{A}$. Then $A$ is said to have \emph{continuous trace} (or be continuous-trace) if for each $t \in \hat{A}$ there exist an open set $U \subset \hat{A}$ containing $t$ and an element $a \in A$ such that $s(a)$ is a rank-one projection for every $s \in U$.
\end{definition}

For an introduction to graph$C^*$- algebras, please see \cite{Raeburn}. The reader who is already familiar with graph $C^*$-algebras may pass over the following standard definitions. 

\begin{definition}
A \emph{(directed) graph} $E$ is an ordered quadruple $E = (E^0,E^1,r,s)$, where the $E^0$ and $E^1$ are countable sets called the \emph{vertices} and \emph{edges}, and $r,s:E^1 \to E^0$ are maps called the \emph{range} and \emph{source} maps. 

A vertex $v$ is called an \emph{infinite receiver} if there are infinitely many edges in $E^1$ with range $v$; a vertex is called a \emph{source} if it receives no sources. A vertex is \emph{regular} if it is neither an infinite receiver or a source; otherwise, it is called \emph{singular}. A graph is \emph{row-finite} if it has no infinite receivers and has \emph{no sources} if every vertex receives an edge. A \emph{cycle} in a directed graph is a path $\lambda \in E^* \setminus E^0$ with $r(\lambda)=s(\lambda)$; a \emph{simple cycle} is a cycle $\lambda$ which does not contain another cycle. An \emph{entrance} to the cycle $\lambda=e_1 \ldots e_n$ is an edge $e$ with $r(e)=e_k$ and $e \neq e_k$. 

The finite path space $E^*$ consists of all finite sequences $e_1 \ldots e_n$ in $E^1$ such that $s(e_i)=r(e_{i+1})$ for $i=1,\ldots,n-1$. The range of the path $e_1 \ldots e_n$ is defined to be $r(e_1)$ and its source is $s(e_n)$. If $\mu=e_1 \ldots e_n$ is a finite path, then we define the length to be $n$ and write $|\mu|=n$. The vertices are included in the finite path space as the paths of length zero. If $\lambda = e_1 \ldots e_n$ and $\mu = f_1 \ldots f_m$ are finite paths with $s(\lambda)=r(\mu)$, we can concatenate them to from $\lambda \mu = e_1 \ldots e_n f_1 \ldots f_m \in E^*$. The \emph{infinite path space} is $E^\infty = \{ e_1 e_2 \ldots | s(e_i) =r(e_{i+1}) \forall i \geq 1 \}$. If $\lambda = e_1 \ldots e_n \in E^*$ and $x = f_1 \ldots \in E^\infty$, then $\lambda x = e_1 \ldots e_n f_1 \ldots \in E^\infty$. The range of $x=e_1 e_2 \ldots \in E^\infty$ is defined as $r(x):=r(e_1)$. The shift map $\sigma:E^\infty \to E^\infty$ removes the first edge from an infinite path: $\sigma(e_1 e_2 \ldots) = e_2 e_3 \ldots$. Composing $\sigma$ with itself yields powers $\sigma^2, \sigma^3,\ldots$. 
\end{definition}

\begin{definition}
Let $E$ be a directed graph. Then the \emph{graph $C^*$-algebra} of $E$, denoted $C^*(E)$, is the universal $C^*$-algebra generated by projections $\{p_v: v \in E^0\}$ and partial isometries $\{s_e: e \in E^1\}$ satisfying the following Cuntz-Krieger relations: 

\begin{enumerate}
\item $s_e^* s_e = p_{s(e)}$ for any $e \in E^1$; 
\item $s_e s_e^* \leq p_{r(e)}$ for any $e \in E^1$; 
\item $s_e^* s_f=0$ for distinct $e,f \in E^1$;
\item If $v$ is regular, then $p_v = \sum_{r(e)=v} s_e s_e^*$. 
\end{enumerate} 
\end{definition}

We also include the basic definitions for groupoids. A concise definition of a groupoid is a small category with inverses; we include a more detailed definition.

A \emph{groupoid} is a set $G$ along with a subset $G^{(2)} \subset G \times G$ of \emph{composable pairs} and a two functions: composition $\circ: G^{(2)} \to G$ (written $(\alpha,\beta) \to \alpha \beta)$ and an involution $^{-1}: G \to G$ (written $\gamma \to \gamma^{-1}$) such that the following hold: 
\begin{itemize}
\item[(i)] $\gamma(\eta \zeta) = (\gamma \eta) \zeta$ whenever $(\gamma,\eta),(\eta,\zeta) \in G^{(2)}$; 
\item[(ii)] $(\gamma,\gamma^{-1}) \in G^{(2)}$ for all $\gamma \in G$, and $\gamma^{-1}(\gamma \eta) = \eta$ and $(\gamma \eta) \eta^{-1} = \gamma$ for $(\gamma,\eta) \in G^{(2)}$. 
\end{itemize}
Elements satisfying $g = g^2 \in G$ are called \emph{units} of $G$ and the set of all such units is denoted $G^{(0)} \subset G$ and called the \emph{unit space} of $G$. There are maps $r,s: G \to G^{(0)}$ defined by 
\[
r(\gamma) = \gamma \gamma^{-1} \qquad \qquad s(\gamma) =\gamma^{-1} \gamma \]
that are called, respectively, the \emph{range} and \emph{source} maps. These maps orient $G$ as category, with units serving as objects: $(\alpha,\beta) \in G^{(2)}$ if and only if $s(\alpha)=r(\beta)$. For a given unit $u \in G^{(0)}$ there is an associated group $G(u) = \{\gamma \in G: r(\gamma) = s(\gamma) = u \}$; this is called the \emph{isotropy} or \emph{stabilizer group} of $u$. The union of all isotropy groups in $G$ forms a subgroupoid of $G$ called $\operatorname{Iso}(G)$, the \emph{isotropy bundle} of $G$. A groupoid is called \emph{principal} (or an \emph{equivalence relation}) if $\operatorname{Iso}(G) = G^{(0)}$; that is, if no unit has non-trivial stabilizer group. 

A topological groupoid is a groupoid $G$ endowed with a topology so that the composition and inversion operations are continuous (the domain of $\circ$ is equipped with the relative product topology). A topological groupoid is \emph{\'etale} if the topology is locally compact and the range and source maps are local homeomorphisms. All the groupoids we encounter in this paper will be Hausdorff and second countable. Note that if $G$ is \'etale then each range fiber $r^{-1}(u)$ is a discrete in the relative topology (likewise for source fibers). Hence a compact subset of $G$ meets a given range fiber at most finitely many times. 

In order to define a $C^*$-algebra from an \'etale groupoid $G$, it is necessary to specify a $*$-algebra structure on $C_c(G)$. This is given by 
\[
(f*g)(\gamma) = \sum_{(\alpha,\beta) \in G^{(2)}: \alpha \beta = \gamma} f(\alpha) g(\beta);\]
compactness of supports ensures that this sum gives a well-defined element of $C_c(G)$. (Really the important thing here is that the counting measures on range fibers form a \emph{Haar system}, which is necessary for any topological groupoid to define a $C^*$-algebra; see \cite{Renault}.) We do not include all details on how to put a norm on $C_c(G)$, these can be found in \cite{Renault}. In brief, there are two distinguished $C^*$-norms $||\cdot||$,$|| \cdot||_r$ on $C_c(G)$ and completing in these yields the full groupoid $C^*$-algebra $C^*(G)$ and the reduced groupoid $C^*_r(G)$, respectively. Our interest in groupoid $C^*$-algebras will strictly be as an alternate model for graph $C^*$-algebras. 

\begin{definition}
Let $E$ be a graph and let $E^\infty$ denote its infinite path space. Then the \emph{path groupoid} of $E$ is 
\[
G_E = \{(x,n,y) \in E^\infty \times \Z \times E^\infty:  \exists p,q \in \N \text{ such that } \sigma^p x = \sigma^q y \text{ and } p-q=n \} \]
The groupoid operations are $(x,n,y) (y,m,z)=(x,m+n,z)$ and $(x,n,y)^{-1}=(y,-n,x)$. The unit space is identified with $E^\infty$ via the mapping $x \mapsto (x,0,x)$, so that the range and source maps are given by $r(x,n,y) = x$ and $s(x,n,y)=y$. 
\end{definition}

\begin{remark} If $G=G_E$, then the isotropy group of an infinite path $x$ is either trivial ($\sigma^p x = \sigma^q x$ implies $p=q$) or infinite cyclic (in which case $x = \alpha (\lambda^\infty)$ for some finite path $\alpha$ and cycle $\lambda$). 
\end{remark}

The topology on $G_E$ is generated by basic open sets of the form
\[
Z(\alpha,\beta) = \{ (\alpha z, |\alpha|-|\beta|, \beta z) \in G_E : r(z) = s(\alpha) \} \]
where $\alpha,\beta \in E^*$ with $s(\alpha)=s(\beta)$. The topology defined above restricts to the relative product topology on $G_E^{(0)} = E^\infty \subset \prod_{\mathbb{N}} E^1$, if we treat $E^1$ as a discrete space. We will refer to this topology on $E^\infty$ using basic compact-open sets of the form $Z(\alpha) = \{ \alpha x | x \in E^\infty, r(x) = s(\alpha) \}$.

It is noted in \cite{KPRR} that $Z(\alpha,\beta) \cap Z(\gamma,\delta) = \emptyset$ unless $(\alpha,\beta)=(\gamma \epsilon, \delta \epsilon)$ or vice versa.  This topology makes $G_E$ into an \'etale groupoid (\cite[Prop. 2.6]{KPRR}), because the restriction of the range map to the basic sets is a homeomorphism, and furthermore each basic set is compact. Thus $G_E$ has a canonical Haar system $\{\lambda_x\}_{x \in E^\infty}$ consisting of counting measures on the source fibers. Because $G_E$ is an \'etale groupoid, it has an groupoid $C^*$-algebra $C^*(G_E)$ (\cite{Renault}). The following theorem has been modified from its original statement to fit our orientation convention.

\begin{theorem}[{\cite[Thm. 4.2]{KPRR}}]
For any row-finite graph with no sources $E$, we have $C^*(E) \cong C^*(G_E)$ via an isomorphism carrying $s_e$ to $\mathbf{1}_{Z(e,s(e))} \in C_c(G)$ and $p_v$ to $\mathbf{1}_{Z(v,v)}$. 
\end{theorem}

(It is a fact that $G_E$ is always an \emph{amenable} groupoid, so that we have $C^*(G_E) = C^*_r(G_E)$.) To describe graph $C^*$-algebras with continuous trace, we need to know when the isotropy groups $G(u)$ vary continuously with respect to the unit $u \in G^{(0)}$. First the topology on the set of isotropy groups has to be defined. 

\begin{definition}[{\cite{RaeWil}}]
Let $X$ be a topological space. Consider the collection $F(X)$ of all closed subsets of $X$; the \emph{Fell topology} on $F(X)$ is defined by the requirement that a net $(Y_i)_{i \in I} \subset F(X)$ converges to $Y \in F(X)$ exactly when
\begin{enumerate}
\item if elements $y_i$ are chosen in $Y_i$ such that $y_i \rightarrow z$, then $z$ belongs to $Y$, and 
\item for any element $y \in Y$ there is a choice of elements $y_i \in Y_i$ (possibly taking a subnet of $(Y_i)$ and relabeling) such that $y_i \to y$. 
\end{enumerate}
We say that $G$ has \emph{continuous isotropy} if the isotropy map $G^{(0)} \to F(G)$ defined by $x \mapsto G(x)$ is continuous.
\end{definition}

We will not need to handle the Fell topology directly in this paper, thanks to the following result which describes continuous isotropy for graph algebras. 

\begin{theorem}[\cite{Goehle2}]
Let $E$ be a row-finite graph with no sources. Then $G_E$ has continuous isotropy if and only if no cycle in $E$ has an entrance. 
\end{theorem}

\begin{definition}
A topological groupoid $G$ is \emph{proper} if the orbit map $\Phi_G:G \to G^{(0)} \times G^{(0)}$ given by $g \to (r(g),s(g))$ is proper (where the codomain is equipped with the relative product topology). 
\end{definition}

\begin{definition}
Let $G$ be a groupoid. Let $\pi_R: G \to G^{(0)} \times G^{(0)}$ be given by $\pi_R(g) = (r(g),s(g))$. Then the \emph{orbit groupoid} of $G$, denote by $R_G=R$, is the image of $\pi_R$, where the groupoid operations are 
\[
(u,v) (v,w)=(u,w)\]
\[
(u,v)^{-1}=(v,u). \]
The unit space of $R$ is identified with the unit space of $G$. The range and source maps are naturally identified with the projections onto the first and second factors. 
\end{definition}

\begin{definition} The topology on $R=R_G$ is the quotient topology induced by the above map $\pi_R: G \to R$. 
\end{definition}

\begin{remark} If the groupoid $G$ is principal, in the sense that $G(x)=\{x\}$ for every unit $x \in G^{(0)}$, then the map $\pi_R$ is a groupoid isomorphism. If $G$ is an arbitrary topological groupoid, it need not be the case that the quotient topology makes $R$ into a topological groupoid. However, for graph and $k$-graph groupoids (we will see these later on), this issue never arises and $R$ will always be a topological groupoid. The following commutative diagram serves to keep the relevant groupoids and spaces distinct. Note that as a set map, $\Phi_R$ is just an inclusion. However, $R$ carries a different topology from the product topology on $G^{(0)} \times G^{(0)}$, so we distinguish between the two. 
\end{remark} 

\begin{center}
\begin{tikzpicture}
\matrix (m) [matrix of math nodes, row sep=3em, column sep=4em, minimum width=2em] 
{
G &  \\
& G^{(0)} \times G^{(0)} \\
R & \\}; 
\path [-stealth]
(m-1-1) edge node [left] {$\pi$} (m-3-1)
            edge  node [above] {$\Phi_G$} (m-2-2)
 (m-3-1) edge node [below] {$\Phi_R$} (m-2-2); 
\end{tikzpicture}
\end{center}

\begin{remark}
In some sources (such as \cite{MRW3}) the orbit groupoid is denoted by $R=\mathcal{G}/\mathcal{A}$, to indicate that it is the quotient of a groupoid by the (in this case, abelian) isotropy bundle. By $\Sigma^{(0)}$ we denote the collection of closed subgroups of $G$, equipped with the Fell topology. 
\end{remark}

\begin{theorem}{\cite[Thm 1.1]{MRW3}}
Let $G$ be a second-countable locally compact Hausdorff groupoid with unit space $G^{(0)}$, abelian isotropy, and Haar system $\{\lambda^u\}_{u \in G^{(0)}}$. Then $C^*(G,\lambda)$ has continuous trace if and only if
\begin{enumerate}
\item the stabilizer map $u \mapsto G(u)$ is continuous from $G^{(0)}$ to $\Sigma^{(0)}$;
\item the action of $R$ on $G^{(0)}$ is proper. 
\end{enumerate}
\end{theorem}

\section{Continuous-trace graph algebras}

The path groupoid of a directed graph $E$ is made of infinite paths, and the open sets are described by finite path prefixes. It is unsurprising then that the characterization of proper path groupoids is stated in terms of a certain finiteness condition on paths. \textbf{For this section, the standing assumption is that $E$ is a row-finite graph with no sources}. 

\begin{definition}
Let $v$ and $w$ be vertices in a directed graph $E$. An \emph{ancestry pair} for $v$ and $w$ is a pair of paths $(\lambda,\mu)$ such that $r(\lambda)=v$, $r(\mu)=w$, and $s(\lambda)=s(\mu)$. A \emph{minimal ancestry pair} is an ancestry pair $(\lambda,\mu)$ such that $(\lambda,\mu)=(\lambda' \nu, \mu' \nu)$ implies that $\nu=s(\lambda) = s(\mu)$. An ancestry pair$(\lambda,\mu)$ is \emph{cycle-free} if neither path contains a cycle. The graph $E$ has \emph{finite ancestry} if for every pair of vertices (not necessarily distinct) $v,v'$ has at most finitely many cycle-free minimal ancestry pairs. 
\end{definition}
\begin{remark}
Note that it is not necessary that any two vertices of $E$ have an ancestry pair in order for $E$ to have finite ancestry. 
\end{remark}

\begin{lemma}
Let $(\alpha,\beta)$ and $(\gamma,\delta)$ be two cycle-free minimal ancestry pairs. Then $\pi_R(Z(\alpha,\beta)) \cap \pi_R(Z(\gamma,\delta)) = \emptyset$ unless there is a simple cycle $\lambda$ and either factorizations $\alpha= \gamma \alpha'$, $\delta = \beta \delta'$ and $\lambda = \alpha' \delta'$ or factorizations $\gamma = \alpha \gamma'$, $\beta = \delta \beta'$ and $\lambda = \gamma' \beta'$. If $(\alpha,\beta)$ is a minimal cycle-free ancestry pair, and if no cycle of $E$ has an entrance, then there are finitely many minimal cycle-free ancestry pairs $(\gamma,\delta)$ such that $\pi_R(Z(\alpha,\beta)) \cap \pi_R(Z(\gamma,\delta)) \neq \emptyset$
\end{lemma}

\begin{proof}
We have that $\pi_R(Z(\alpha,\beta)) = \{(\alpha z, \beta z) \} \subset E^\infty \times E^\infty$ and $\pi_R(Z(\gamma,\delta)) = \{ (\gamma w, \delta w) \}$. Suppose that $(\alpha z, \beta z) = (\gamma w, \delta w)$ as pairs of infinite paths. If $|\alpha| \geq |\gamma|$ and $|\beta| \geq |\delta|$, then $\alpha = \gamma \nu$ and $\beta = \delta \nu'$, where $\nu$ and $\nu'$ are initial segments of the infinite path $w$. If $\nu = \nu'$, this contradicts minimality of the ancestry pair $(\alpha,\beta)$. If $\nu \neq \nu'$, then (assuming WLOG that $\nu \subset \nu'$) we can note that $s(\nu) = s(\alpha)=s(\beta) = s(\nu')$, so that $\nu'$ contains a cycle, hence that $\beta$ contains a cycle, contradicting the assumption that $(\alpha,\beta)$ is a cycle-free ancestry pair. 

We must have WLOG that $|\alpha| > |\gamma|$ and $|\beta|< |\delta|$. We can factor $\alpha = \gamma \alpha'$ and $\delta = \beta \delta'$. Then $\lambda = \alpha' \delta'$ is a cycle with range/source equal to $s(\gamma)=s(\delta)$. This cycle must be simple because neither $\alpha' \subset \alpha$ nor $\delta' \subset \delta$ contains a cycle. This establishes the first claim in the lemma.

 If no cycle has an entry then any vertex is the source of at most one simple cycle. If $(\alpha,\beta)$ is a cycle-free minimal ancestry pair, and no cycle of $E$ has an entrance, then there is at most one simple cycle at $s(\alpha)=s(\beta)$. Thus there are only finitely many factorizations of the above form. 
\end{proof}

\begin{lemma}
Let $(\lambda,\mu)$ be an ancestry pair. Then there is a unique minimal ancestry pair $(\alpha,\beta)$ such that $Z(\lambda,\mu) \subset Z(\alpha,\beta)$ in $G_E$. If $(\lambda,\mu)$ is cycle-free, then so is $(\alpha,\beta)$. 
\end{lemma}
\begin{proof}
Let $\epsilon$ be maximal (with respect to length) such that $(\alpha, \beta) = (\alpha' \epsilon, \beta' \epsilon)$. Then $Z(\alpha,\beta) \subset Z(\alpha',\beta') \neq \emptyset$. As mentioned in \cite{KPRR}, $Z(\lambda,\mu) \cap Z(\gamma,\delta)$ implies that $(\lambda,\mu)$ factors as $(\gamma \epsilon, \delta \epsilon)$ or vice versa. Thus $Z(\alpha,\beta)$ is contained in $Z(\alpha',\beta')$ and not in $Z(\gamma,\delta)$ for any other minimal ancestry pair $(\gamma,\delta)$. The second claim follows from the construction. 
\end{proof}

\begin{lemma}
Let $E$ be a directed graph in which no cycle has an entrance, and let $(\alpha,\beta)$ be an ancestry pair such that $\alpha$ or $\beta$ contains a cycle. Then there is a cycle-free minimal ancestry pair $(\lambda,\mu)$ such that $\pi_R(Z(\lambda,\mu))$ contains $\pi_R(Z(\alpha,\beta))$. 
\end{lemma}

\begin{proof}
Suppose that $\alpha$ contains a cycle, so that $\alpha = \alpha' \lambda \lambda'$, where $\alpha'$ does not contain a cycle, $\lambda$ is a cycle, and $\lambda'$ is an initial segment of $\lambda$. Let $\lambda = \lambda' \lambda''$. Note that $\alpha$ must factor in this way because of the condition that no cycle has an entrance. Then $\pi_R(Z(\alpha,\beta)) = \{(\alpha x, \beta x) : x \in E^\infty, r(x) = s(\alpha) \}$. Again by the condition that no cycle has an entrance, we have that the only infinite path with range equal to $s(\alpha)$ is $\lambda'' (\lambda^\infty)$. Thus $\pi_R(Z(\alpha,\beta))= \pi_R(Z(\alpha' \lambda', \beta))$. We can perform a similar operation to get rid of any cycles from $\beta$, obtaining a cycle free ancestry pair $(\alpha', \beta')$ such that $\pi_R(Z(\alpha,\beta)) = \pi_R(Z(\alpha',\beta'))$. By the previous lemma we can find a cycle-free minimal ancestry pair $(\lambda,\mu)$ such that $Z(\alpha',\beta') \subset Z(\lambda,\mu)$. The lemma is established on applying $\pi_R$ to this containment. 
\end{proof}

\begin{theorem} \label{main-guy}
Let $E$ be a row-finite directed graph with no sources. Then $C^*(E)$ has continuous trace if and only if both 
\begin{enumerate}
\item no cycle of $E$ has an entrance, and
\item $E$ has finite ancestry. 
\end{enumerate}
\end{theorem}

\begin{proof}
Suppose no cycle of $E$ has an entrance and $E$ has finite ancestry. We show that $\Phi_R$ is proper. The collection of sets of the form $Z(v) \times Z(w) \subset E^\infty \times E^\infty$ form a compact-open cover for $E^\infty \times E^\infty$. Thus it suffices to prove that $\Phi_R^{-1}(Z(v) \times Z(w))$ is compact for any vertices $v$ and $w$. It is not hard to see that \[ \Phi_G^{-1}(Z(v) \times Z(w)) = \bigcup_{(\alpha,\beta) \in \mathcal{M}} Z(\alpha,\beta),\] where $\mathcal{M}$ is the set of all minimal ancestry pairs for $v$ and $w$. We can partition these pairs into two families: $\mathcal{C}$, the set of minimal ancestry pairs for $v$ and $w$ containing cycles and $\mathcal{D}$, the set of cycle-free minimal ancestry pairs. 

By definition, we have that $\Phi_R^{-1}(Z(v) \times Z(w)) = \pi_R( \Phi^{-1}(Z(v) \times Z(w)))$. Thus we can write 
\[
\pi_R(  \Phi^{-1}(Z(v) \times Z(w)) ) = \pi_R\left( \bigcup_{(\alpha,\beta) \in \mathcal{C}} Z(\alpha,\beta) \right) \cup \pi_R \left( \bigcup_{(\lambda,\mu) \in \mathcal{D}} Z(\lambda,\mu) \right) .\]

But for each $(\alpha,\beta) \in \mathcal{C}$, we have that $\pi_R(Z(\alpha,\beta)) \subset \pi_R(Z(\lambda,\mu))$ for some $(\lambda,\mu) \in \mathcal{D}$. Thus $\Phi_R^{-1}(Z(v) \times Z(w)) = \cup_{(\lambda,\mu) \in \mathcal{D}} \pi_R(Z(\lambda,\mu))$. Each $\pi_R(Z(\lambda,\mu))$ is compact by the continuity of $\pi_R$ and compactness of $Z(\lambda,\mu)$, and $\mathcal{D}$ is finite by the assumption that $E$ has finite ancestry. Thus $\Phi_R$ is a proper map and the $C^*$-algebra has continuous trace by \cite{MRW3}. 

Now suppose that $C^*(E)$ has continuous trace. Then the isotropy groups of $G$ must vary continuously so no cycle of $E$ has an entrance. So we only need to show that $E$ has finite ancestry. Suppose that $v$ and $w$ are two vertices and let $\{(\alpha_k,\beta_k): k \in \mathcal{A} \}$ be an enumeration of the cycle-free minimal ancestry pairs for $v$ and $w$. As in the proof of sufficiency, we can write $\Phi_R^{-1}(Z(v) \times Z(w)) = \cup_k \pi_R(Z(\alpha_k, \beta_k))$. By properness of $\Phi_R$ we must be able to extract a finite subcover. However for any finite subset $B \subset A$, the set $B'=\{k \in A: \pi_R(Z(\alpha_k,\beta_k)) \cap (\cup_{j \in B} \pi_R(Z(\alpha_j,\beta_j))) \neq \emptyset \}$ is finite by Lemma 1. This implies that $A$ is finite. Thus $E$ has finite ancestry. 
\end{proof}

\section{Arbitrary graphs}

The previous theorem is given only in the context of row-finite graphs with no sources. In this section we will remove the requirement that all graphs be row-finite and have no sources, by use of the Drinen-Tomforde desingularization. 

\begin{definition}
A \emph{tail} at the vertex $v$ is an infinite path with range $v$. 
\end{definition}

Briefly, the Drinen-Tomforde desingularization adds a tail to each singular vertex. If the singular vertex $v$ is an infinite receiver, it takes all the edges with range $v$ and redirects each to a different vertex on the infinite tail. This produces a new graph $\tilde{E}$ which has no singular vertices. For details, see \cite{DrinTom} or \cite{Raeburn}. Note that we have reversed the edge orientation of \cite{DrinTom}, to fit with the higher-rank graphs considered in the next section. 

\begin{theorem}{\cite[Thm 2.11]{DrinTom}}
Let $E$ be an (arbitrary) directed graph. Let $\tilde{E}$ be a desingularization for $E$. Then $C^*(E)$ embeds in $C^*(\tilde{E})$ as a full corner, so that $C^*(E)$ is Morita equivalent to $C^*(\tilde{E})$. 
\end{theorem}

The basic technical lemma needed is a bijection between finite paths in a singular graph and certain finite paths in its desingularization. (We've omitted the part about infinite paths.)
\begin{lemma}{\cite[Lemma 2.6]{DrinTom}}
Let $E$ be a directed graph and let $\tilde{E}$ be a desingularization. Then there is a bijection \[ 
\phi: E^* \to \{ \beta \in \tilde{E}^*: s(\beta),r(\beta) \in E^0\}.\] The map $\phi$ preserves source and range. 
\end{lemma}

\begin{lemma}
Let $E$ be a directed graph and let $\tilde{E}$ be a desingularization for $E$. Then no cycle of $E$ has an entrance if and only if no cycle of $\tilde{E}$ has an entrance. 
\end{lemma}
\begin{proof}
Suppose that no cycle of $\tilde{E}$ has an entrance. Let $\lambda = e_1 \ldots e_n$ be a cycle in $E$ and let $\tilde{\lambda} = \phi(\lambda) = f_1 \ldots f_m$ be the corresponding path in $\tilde{E}$, with $r(\tilde{\lambda})=r(\lambda)$ and $s(\tilde{\lambda}) = s(\lambda)$. Suppose that $e$ is an edge in $E$ with $r(e)=r(e_k)$ and yet $e \neq e_k$. Then $\tilde{e} = \phi(e)$ is a path in $\tilde{E}$ with $r(\tilde{e}) = r(\phi(e_k))$ and yet $\tilde{e} \neq \tilde{e}_k$ (here we are using the fact that $\phi$ is a bijection). Thus $\tilde{e}$ is an entrance to the cycle $\tilde{\lambda}$. 

Suppose that no cycle of $E$ has an entrance, and let $\mu = f_1 f_2 \ldots f_n$ be a cycle in $\tilde{E}$. If $s(\mu)$ belongs to $E^0$, then $\phi^{-1}(\mu)$ is a cycle in $E$. Furthermore, we know that no vertex of $E^0$ on $\phi^{-1}(\mu)$ can be singular, because then the cycle $\phi^{-1}(\mu)$ would have an entrance. Thus $\mu$ consists solely of edges in $E$ and does not meet any singular vertices or tails. The only edges in $\tilde{E}$ that meet $\mu$ are images under $\phi$ of edges from $E$, and we know that $\mu$ has no entrances in $E$, so $\mu$ has no entrances. 

Now we show that, under the assumption that no cycle of $E$ has an entrance, no cycle of $\tilde{E}$ can have source on an infinite tail. Suppose that $\mu$ is a cycle in $\tilde{E}$ with source on an infinite tail. Because no infinite tail contains a cycle, we can write $\mu = f_1 \ldots f_k d_1 \ldots d_j$, where $d_1 \ldots d_j$ is the largest path in the infinite tail containing $s(\mu)$ such that $d_1 \ldots d_j$ is contained in $ \mu$. Then $r(d_1)$ must be the vertex to which the infinite tail is attached, i.e. $r(d_1) \in E^0$. Consider the cycle $\mu' = d_1 \ldots d_j f_1 \ldots f_k$. This begins and ends in $E^0$, so it equals $\phi(\lambda)$ for some cycle $\lambda$ in $E$. This cycle cannot meet any singular vertices in $E$ (or else it would have an entrance), so it must be the case that $\lambda = \phi(\lambda)$. But $s(\mu)$ belongs to $\lambda$, contradicting our assumption that $s(\mu)$ belongs to an infinite tail. Combining this with the previous part shows that if no cycle of $E$ has an entrance, then no cycle of $\tilde{E}$ has an entrance. 
\end{proof}

\begin{lemma}
Let $E$ be a directed graph and let $\tilde{E}$ be a desingularization of $E$. Then $\tilde{E}$ has finite ancestry if and only if $E$ has finite ancestry. 
\end{lemma}

\begin{proof}
Suppose that $E$ has finite ancestry and let $v$ be a vertex of $\tilde{E}$. We show that $v$ has finitely many cycle free minimal ancestry pairs by defining an injection from minimal ancestry pairs of $v$ to minimal ancestry pairs of some vertex in $E$. 

If $v$ belongs to $E$, then it must be the case that $s(\alpha)=s(\beta)$ belongs to $E^0$ for any minimal ancestry pair $(\alpha,\beta)$. For otherwise $s(\alpha)$ lies on an infinite tail, with only one edge $f$ leaving $s(\alpha)$, and we could factor a common edge $(\alpha,\beta) = (\alpha' f, \beta' f)$. Thus in the case that $v$ belongs to $E$, we can map $(\alpha,\beta) \mapsto (\phi^{-1}(\alpha),\phi^{-1}(\beta))$. This carries cycle-free minimal ancestry pairs to cycle-free minimal ancestry pairs, so $v$ must have finitely many cycle free minimal ancestry pairs in $F$. 

If $v$ belongs to an infinite tail, let $d_1 \ldots d_j$ be the path from $v$ to the singular vertex $w$ to which the infinite tail is attached. Again, if $(\alpha,\beta)$ is a minimal ancestry pair for $v$, it must terminate in a vertex of $E$. Define a map $(\alpha,\beta) \mapsto (\phi^{-1}(d_1 \ldots d_j \alpha, d_1 \ldots d_j \beta_j))$. As above, this will give us an injection into minimal cycle free ancestry pairs for $w$ in $E$. Hence $v$ has finitely many minimal cycle free ancestry pairs. 

If $\tilde{E}$ has finite ancestry then $E$ trivially has finite ancestry by using $\phi$. 

\end{proof}

\begin{theorem}
Let $E$ be an arbitrary graph. Then $C^*(E)$ has continuous trace if and if both
\begin{enumerate}
\item no cycle of $E$ has an entrance, and
\item $E$ has finite ancestry. 
\end{enumerate}
\end{theorem}

\begin{proof}
Let us begin by fixing a desingularization $\tilde{E}$ of $E$. If no cycle of $E$ has an entrance and $E$ has finite ancestry, then Lemmas 8 and 9 tell us that the same is true of $\tilde{E}$. Then Theorem 4 says that $C^*(\tilde{E})$ has continuous trace. Theorem 5 and the well-known fact that the class of continuous-trace $C^*$-algebras is closed under Morita equivalence then give that $C^*(E)$ has continuous trace. 

Now suppose that $C^*(E)$ has continuous trace. Then $C^*(\tilde{E})$ has continuous trace as in the previous part of the proof. By Theorem 4, we see that no cycle of $\tilde{E}$ has an entrance and $\tilde{E}$ has finite ancestry. Lemmas 8 and 9 again give us that $E$ satisfies the same conditions. 
\end{proof}

\begin{corollary}
If $E$ is a graph with no cycles, then $C^*(E)$ has continuous trace if and only if $E$ has finite ancestry. 
\end{corollary}

This is useful for studying AF algebras. Drinen showed that every AF algebra arises as the $C^*$-algebra of a locally finite pointed directed graph (\cite{Drinen}). Tyler gave a useful complementary result, showing that if $E$ is Bratteli diagram for an AF algebra $A$, then there is a Bratteli diagram $KE$ for $A$ such that (treating the diagrams as directed graphs) $C^*(KE)$ contains $A$ and $C^*(E)$ as complementary full corners (\cite{Tyler}). Thus in particular $A$ and $C^*(E)$ are Morita equivalent. 

\begin{corollary}
Let $A$ be an AF algebra and let $E$ be a Bratteli diagram for $A$, treated as a directed graph. Then $A$ has continuous trace if and only if $E$ has finite ancestry.
\end{corollary}

\begin{example}
Let $A = \bigotimes_{n=1}^\infty M_2(\C)$ be the UHF algebra of type $2^\infty$. The familiar Bratteli diagram for $A$ (after decoration with labels) is

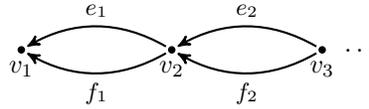
\begin{center}
	\tikzset{every loop/.style={min distance=0.75in, in=165, out=75, looseness=1}}
	\begin{tikzpicture}[->, >=stealth', shorten >=1pt, shorten <=1pt, auto, thick, main node/.style={circle,fill=black,scale=0.3}]
	\node[main node,name=v1] at (0,0) {};
	\node[main node,name=v2]  at (2,0) {};
	\node[main node,name=v3]  at (4,0) {};
	\node at (4.5,0) {$\ldots$};
	\node at (0,-.25) {$v_1$};
	\node at (2,-.25) {$v_2$};
	\node at (4,-.25) {$v_3$};
\path[every node/.style={font=\sffamily\small}]

(v2) edge [bend right] node [above] {$e_1$} (v1)
(v2) edge [bend left] node [below]{$f_1$} (v1)
(v3) edge [bend right] node [above] {$e_2$} (v2)
(v3) edge [bend left] node [below]{$f_2$} (v2);
\end{tikzpicture}
\end{center}

This graph fails to have finite ancestry: for each $k$, we have the cycle-free minimal ancestry pair $(f_1 e_2 f_3 \ldots e_{2k}, e_1 f_2 e_3 \ldots f_{2k})$ for $v_1,v_1$. Thus it does not have continuous trace. (As is well-known, we can actually reach a stronger conclusion, namely that $A$ does not have Hausdorff spectrum; see \cite{Goehle2}.) 
\end{example}

\section{Higher-rank graphs}

In this section we partially extend the results of the previous section to the realm of higher-rank graphs. We have not completely described which higher-rank graph $C^*$-algebras have continuous trace. However, we do characterize the \emph{aperiodic} higher-rank graphs which yield continuous-trace $C^*$-algebras. The jump in combinatorial complexity from the graph to the $k$-graph case is noteworthy. In addition, we provide some negative results regarding the \emph{generalized cycles} of \cite{ES}. In particular, a generalized cycle with entry causes the affiliated vertex projection to be infinite, which cannot happen if the algebra has Hausdorff spectrum. 

\begin{remark}
The semigroup $\N^k$ is treated as a category with a single object, $0$. 
\end{remark}

\begin{definition} 
A \emph{higher-rank graph} (or \emph{$k$-graph}) is a countable category $\Lambda$ equipped with a \emph{degree functor} $d:\Lambda \to \N^k$ which satisfies the following factorization property: if $d(\lambda)=m+n$ for some $m,n \in \N^k$, then $\lambda=\mu \nu$ for some unique $\mu,\nu$ such that $d(\mu)=m$ and $d(\nu)=n$. The \emph{vertices} $\Lambda^0$ of $\Lambda$ are identified with the objects. The elements of $\Lambda$ are referred to as \emph{paths}. For fixed degree $n \in \N^k$, the paths of degree $n$ are denoted by $\Lambda^n$. We refer to paths of degree $0$ as \emph{vertices} in the $k$-graph (each path in $\Lambda$) has a well-defined range vertex and source vertex. 
\end{definition}

We can affiliate a $C^*$-algebra to a higher-rank graph but some additional hypotheses have to be added in order to ensure the result is not trivial. The hypotheses we use are not the weakest set which defines a meaningful $C^*$-algebra; however, they allow us to use the groupoid machinery easily. 
\begin{definition}
A higher-rank graph is \emph{row-finite} if each vertex $v \in \Lambda^0$ and degree $n \in \N^k$, there are only finitely many paths of degree $n$ with range $v$. It is said to have \emph{no sources} if for all $v \in \Lambda^0$ and $n$ there is some path $\lambda$ with $d(\lambda)=n$ and $r(\lambda)=v$. 
\end{definition}

\begin{definition}
Let $\Lambda$ be a row-finite $k$-graph with no sources. Then the higher-rank graph $C^*$-algebra of $\Lambda$, denoted $C^*(\Lambda)$, is the universal $C^*$-algebra generated by a family of partial isometries $\{s_\lambda\}_{\lambda \in \Lambda}$ satisfying: 
\begin{enumerate}
\item $\{s_v: v \in \Lambda^0\}$ is a family of mutually orthogonal projections; 
\item if $\lambda, \mu \in \Lambda$ with $s(\lambda)=r(\mu)$, then $s_\lambda s_\mu = s_{ \lambda \mu}$; 
\item $s_\lambda^* s_\lambda = s_{s(\lambda)}$; 
\item for any $v \in \Lambda^0$ and any degree $n \in \N^k$, we have $s_v = \sum_{\lambda \in \Lambda^n: r(\lambda)=v} s_\lambda s_\lambda^*$. 
\end{enumerate}
\end{definition}

Just as in the graph case, we study continuous-trace higher-rank graph $C^*$-algebras by studying an affiliated groupoid. The following $k$-graph is used to define infinite paths in $k$-graphs. 
\begin{definition}
Let $\Omega_k$ be the category of all pairs $\{(m,n): m \leq n\}$, where $m \leq n$ if $m_i \leq n_i$ for $i=1,\ldots,k$. The composition is given by $(m,n)(n,p)=(m,p)$. The degree functor is given by $d(m,n)=n-m$. The objects are all pairs of the form $(m,m)$. If $\Lambda$ is a $k$-graph, then an \emph{infinite path} in $\Lambda$ is a degree preserving functor $x: \Omega_k \to \Lambda$. The collection of infinite paths in $\Lambda$ is denoted $\Lambda^\infty$. 

Let $\Lambda$ be a $k$-graph and let $x$ be an infinite path in $\Lambda$. For any $p \in \N^k$, we define $\sigma^p x$ to be the infinite path given by  $\sigma^px(m,n)=x(m+p,n+p)$.  The \emph{range} of an infinite path $x \in \Lambda^\infty$ is defined to be $x(0,0)$. If $\lambda \in \Lambda$ and $x \in \Lambda^\infty$ with $s(\lambda)=r(x)$, then there is a unique path $y=\lambda x \in \Lambda^\infty$ such that $\sigma^{d(\lambda)} y = x$ and $y(0,d(\lambda)) = \lambda$. 
\end{definition}

Now we can define the higher-rank version of the path groupoid. As noted in \cite{KP},  by the no sources assumption we know that for every vertex $v \in \Lambda^0$, there is at least one $x \in \Lambda^\infty$ with $r(x)=v$. 

\begin{definition}[\cite{KP}]
Let $\Lambda$ be a $k$-graph. Then the \emph{path groupoid} of $\Lambda$ is \[
 G_\Lambda = \{(x,n,y) \in \Lambda^\infty \times \Z^k \times \Lambda^\infty: \exists p,q \in \N^k \text{ such that } \sigma^p x = \sigma^q y \text{ and } p-q = n \} .\] The groupoid operations are given by $(x,n,y)(y,m,z)=(x,m+n,z)$ and $(x,n,y)^{-1}=(y,-n,x)$. 
\end{definition}

The topology on $G_\Lambda$ is defined in the same way as the topology on $G_E$, for $E$ a graph. The basic open sets are 
\[
Z(\alpha,\beta) = \{ (x,n,y) \in G_\Lambda: \sigma^{d(\alpha)}(x)=\sigma^{d(\beta)}(y), d(\alpha)-d(\beta) = n \} .\]

The topology on $G_\Lambda$ generated by these sets makes it into a locally compact Hausdroff \'{e}tale groupoid over unit space $\Lambda^\infty$. (See \cite[Prop. 2.8]{KP}.) The relative topology on the unit space can be described by open sets of the form $Z(\alpha) =\{x \in \Lambda^\infty: x(0,d(\alpha))=\alpha \}$. 

\begin{theorem}{\cite[Cor. 3.5]{KP}}
Let $\Lambda$ be a row-finite $k$-graph with no sources and let $G_\Lambda$ be its path groupoid. Then $C^*(\Lambda) \cong C^*(G_\Lambda)$. 
\end{theorem}

Therefore we can decide questions about $C^*(\Lambda)$ by studying $G_\Lambda$. Deciding when the path groupoid $G_\Lambda$ has continuously varying isotropy groups is substantially harder than the graph case. 

\begin{definition}
Let $\Lambda$ be a row-finite $k$-graph with no sources. Then $\Lambda$ is said to be \emph{tight} if whenever $g =(x,n,x) \in G_\Lambda(x)$ for some $x \in \Lambda^\infty$, there must exist $\alpha,\beta \in \Lambda$ such that 
\begin{itemize}
\item[(i)] $g \in Z(\alpha,\beta)$ (so that in particular $r(\alpha)=r(\beta)=r(x)$ and $s(\alpha)=s(\beta)$;
\item[(ii)] if $y \in \Lambda^\infty$ and $r(y) = s(\alpha)$, then $\alpha y = \beta y$. 
\end{itemize}
\end{definition}

\begin{remark}
It is not difficult to see that a $1$-graph $E$ is tight if and only if no cycle of $E$ has an entrance. 
\end{remark}

\begin{lemma} \label{tight}
Let $\Lambda$ be a row-finite $k$-graph with no sources. Then $G_\Lambda$ has continuously varying stabilizers if and only if $\Lambda$ is tight. 
\end{lemma}

\begin{proof}
It is shown in \cite{LalMil} that an \'etale groupoid has continuously varying isotropy if and only if the isotropy subgroupoid is open. The lemma follows immediately then from this and the description of the topology on $G_\Lambda$. 
\end{proof}

We modify our definition of ancestry pair to the $k$-graph situation. 

\begin{definition}
Let $\Lambda$ be a row-finite $k$-graph with no sources and let $v,w \in \Lambda^0$ be two vertices. Then an \emph{ancestry pair} for $v,w$ is a pair $(\lambda,\mu) \in \Lambda \times \Lambda$ such that $r(\lambda)=v$, $r(\mu)=w$, and $s(\lambda)=s(\mu)=w$. An ancestry pair $(\lambda,\mu)$ is \emph{minimal} if $(\lambda,\mu)=(\lambda' \nu, \mu' \nu)$ implies $\nu=s(\lambda)$. We say that $\Lambda$ has \emph{strong finite ancestry} if each pair of vertices has at most finitely many minimal ancestry pairs. 
\end{definition}

\begin{remark}
Strong finite ancestry implies finite ancestry in the $1$-graph case. In fact, a $1$-graph $E$ having strong finite ancestry is equivalent to $E$ having finite ancestry and $E$ having no directed cycles. 
\end{remark}

By $R_\Lambda$ we denote the orbit groupoid of $G_\Lambda$ (again equipped with the quotient topology from $\gamma \mapsto (r(\gamma),s(\gamma))$).

\begin{lemma} \label{strong-finite}
Suppose that $\Lambda$ is a row-finite $k$-graph with no sources. If $\Lambda$ has strong finite ancestry, then $R_\Lambda$ is proper. 
\end{lemma}

\begin{proof}
We adopt the notation of Theorem \ref{main-guy}. As in the proof of Theorem \ref{main-guy}, we see that \[
 \Phi^{-1}_R(Z(v) \times Z(w)) = \cup_{(\alpha,\beta) \in \mathcal{M}} \pi_R(Z(\alpha,\beta)) \]
 where $\mathcal{M}$ is the set of minimal finite ancestry pairs for $v$ and $w$. Strong finite ancestry then implies that $\Phi_R$ is proper, so that $R_\Lambda$ is proper. 
\end{proof}

The following lemma is used to show that finite ancestry is necessary for a (strongly aperiodic) $k$-graph to yield a $C^*$-algebra with continuous trace. A $k$-graph is called \emph{strictly aperiodic} if for any $x \in \Lambda^\infty$ and $p,q \in \mathbb{N}^k$, $\sigma^p x = \sigma^q x$ implies $p=x$ (that is, if there are no periodic infinite paths in $\Lambda$). This implies that every isotropy group in $G_\Lambda$ is trivial, so that the map $G_\Lambda \mapsto R_\Lambda$ is an isomorphism. 
\begin{lemma} \label{overlap}
Let $\Lambda$ be a strictly aperiodic row-finite $k$-graph with no sources. Let $(\alpha,\beta)$ and $(\gamma,\delta)$ be two distinct minimal ancestry pairs in $\Lambda$. Then $Z(\alpha,\beta) \cap Z(\gamma,\delta)= \emptyset$. 
\end{lemma}
\begin{proof}\emph{Claim}: It suffices to show that if $Z(\alpha, \beta) \cap Z(\gamma,\delta) \neq \emptyset$, then either $d(\alpha) \geq d(\gamma)$ and $d(\beta) \geq d(\delta)$, or $d(\gamma) \geq d(\alpha)$ and $d(\delta) \geq d(\beta)$. For suppose that $(\alpha z, n, \beta z) = (\gamma w, n, \delta w)$, $d(\alpha) \geq d(\gamma)$ and $d(\beta) \geq d(\delta)$. Then $\alpha z = \gamma \alpha' z  = \gamma w$, so that $\alpha' z = w$. We also have $\beta z = \delta \beta' z = \delta w$, so that $\beta' z = w$. If $d(\alpha')=d(\beta')$, then this shows that $\alpha' = \beta'$, so that $(\alpha,\beta) = (\gamma \alpha', \beta \alpha')$, contradicting minimality. If $d(\alpha') \neq d(\beta')$, then $\sigma^{d(\alpha')} w = z = \sigma^{d(\beta')} w$, so that $w$ is periodic, against hypothesis. The other case follows by symmetry. This establishes the claim. 

If the intersection is nonzero then as above we have $(\alpha z, n, \beta z) = (\gamma w, n, \delta w)$ for some $z,w \in \Lambda^\infty$. Moreover, $n = d(\alpha)-d(\beta)=d(\gamma)-d(\delta)$, so that $d(\alpha)-d(\gamma)=d(\beta)-d(\delta)$. Thus if $d(\alpha) \geq d(\gamma)$, we also have $d(\beta) \geq d(\delta)$, reducing to the claim. Thus we assume that $d(\alpha)_1 - d(\gamma)_1 = d(\beta)_1 - d(\delta)_1 > 0$ and $d(\gamma)_2 - d(\alpha)_2 = d(\delta)_2-d(\beta)_2 > 0$ (*). From the equation $\alpha z = \gamma w$ and the inequality $d(\alpha)_1 > d(\gamma)_1$, we obtain $\alpha(d(\gamma)_1,d(\alpha)) z = \gamma( d(\gamma)_1, d(\gamma)) w$; call this path $x_1$. Similarly we have $\beta(d(\beta)_1,d(\beta))z = \delta(d(\delta_1),d(\delta)) w$. The conditions (*) imply that $d(\alpha)-d(\gamma)_1 + d(\delta)-d(\delta)_1 = d(\beta) - d(\delta)_1 + d(\gamma)-d(\gamma)_1$. 

Now iterating the shift we obtain \[ \sigma^{d(\delta)-d(\delta)_1} z = \sigma^{d(\alpha)-d(\gamma)_1 + d(\delta)-d(\delta)_1} x_1 = \sigma^{d(\beta) - d(\delta)_1 + d(\gamma)-d(\gamma)_1} x_1 = \sigma^{d(\beta)-d(\delta)_1} w.\]

Similarly we obtain 
\[
\sigma^{d(\alpha)-d(\alpha)_2} w = \sigma^{d(\gamma)-d(\gamma)_2} z.\]

Now we can write 
\[
\sigma^{d(\delta)-d(\delta)_1} \sigma^{d(\alpha)-d(\alpha)_2} w = \sigma^{d(\delta)-d(\delta)_1} \sigma^{d(\gamma)-d(\gamma)_2} z \] \[ = \sigma^{d(\gamma)-d(\gamma)_2 + d(\delta)-d(\delta)_1} z =\sigma^{d(\gamma)-d(\gamma)_2} \sigma^{d(\beta)-d(\delta)_1} w.\]

Now by strict aperiodicity we must have $d(\delta) - d(\delta)_1 + d(\alpha)-d(\alpha)_2 = d(\gamma)-d(\gamma)_2 + d(\beta)-d(\delta)_1$. Compare second coordinates in the degrees; we must have $d(\delta)_2$ on the left and $d(\beta)_2$ on the right. This gives $d(\delta)_2 = d(\beta)_2$, a contradiction. 
\end{proof}

\begin{theorem} \label{k-big}
Let $\Lambda$ be a row-finite $k$-graph with no sources. 
\begin{itemize}
\item[(i)] if $\Lambda$ has strong finite ancestry and is tight, then $C^*(\Lambda)$ has continuous trace. 
\item[(ii)] If $C^*(\Lambda)$ has continuous trace and $\Lambda$ is strictly aperiodic, then $\Lambda$ has strong finite ancestry. 
\end{itemize}
\end{theorem}

\begin{proof}
(i): As $\Lambda$ is tight, we have that $G_\Lambda$ has continuous isotropy as in Lemma \ref{tight}. Lemma \ref{strong-finite} implies that $R_\Lambda$ is proper. Thus \cite[Thm. 1.1]{MRW3} implies that $C^*(\Lambda)$ has continous trace. 

(ii): Because $\Lambda$ is strictly aperiodic, we can identify the groupoids $G_\Lambda$ and $R_\Lambda$. Let $v$ and $w$ be two vertices of $\Lambda$; then \[ \Phi^{-1}(Z(v) \times Z(w)) = \bigcup_{(\alpha,\beta) \in \mathcal{M}} Z(\alpha,\beta)\] as in the proof of Theorem \ref{main-guy}. Lemma \ref{overlap} implies that the sets $Z(\alpha,\beta)$ are pairwise disjoint and open. Thus $\mathcal{M}$ must be finite by compactness of $\Phi^{-1}(Z(v) \times Z(w))$, which implies that $\Lambda$ has strong finite ancestry. 
\end{proof}

\begin{remark}
Theorem \ref{k-big} is not as complete as Theorem \ref{main-guy}; a complete description of those $k$-graphs which define continuous-trace $C^*$-algebras like Theorem \ref{main-guy} seems out of reach. This is because the definition of ``cycle-free'' path in a $k$-graph does not readily carry over or generalize to the $k$-graph case.

Generally the conditions discussed on $k$-graphs in the literature (such as the aperiodicity condition in \cite{KP}) imply that the interior of the isotropy bundle of $G_\Lambda$ coincides with $G_\Lambda^{(0)}$ (the interior of the isotropy bundle always includes the unit space). Our definition of tight $k$-graphs is in some sense the opposite of this--if $\Lambda$ is tight then the interior of the isotropy bundle is the entire isotropy bundle. Strict aperiodicity is amounts to saying that the isotropy bundle coincides with the unit space. 
\end{remark}

Desingularization is in general more complicated for higher-rank graphs, so it seems perhaps unlikely that this could be easily extended to higher-rank graphs with sources. However, we can give some necessary conditions for a $k$-graph to satisfy in order that its $C^*$-algebra have continuous trace. The following definition is modified somewhat from \cite{ES}

\begin{definition}[\cite{ES}]
Let $\Lambda$ be a row-finite graph with no sources. Then a pair $(\lambda,\mu) \in \Lambda \times \Lambda$ is a \emph{generalized cycle} if $r(\lambda)=r(\mu)$ and $s(\lambda)=s(\mu)$ and $Z(\lambda) \subset Z(\mu)$. We say that a generalized cycle $(\lambda,\mu)$ \emph{has an entrance} if $(\mu,\lambda)$ is \emph{not} a generalized cycle. (That is, if $Z(\lambda) \subsetneq Z(\mu)$.) 
\end{definition}

Recall that a projection $p$ in a $C^*$-algebra $A$ is \emph{infinite} if there exists $v \in A$ with $v^*v = p$ and $vv ^* < p$; that is, if it is Murray-von Neumann equivalent to a proper subprojection of itself. 

\begin{lemma}{\cite[Cor. 3.8]{ES}}
If $\Lambda$ contains a generalized cycle with entrance, then $C^*(\Lambda)$ contains an infinite projection. 
\end{lemma}

The following simple observation is probably not new, but is proven here for ease of reference. 

\begin{lemma}
If $A$ is a $C^*$-algebra containing an infinite projection, then $A$ does not have continuous trace. 
\end{lemma}

\begin{proof}
Let $p$ be a projection in $A$ with a proper subprojection $q$ such that $p \sim q$. Take an irreducible representation $\pi:A \to B(H)$  such that $\pi(p-q) \neq 0$. Then $\pi(q) < \pi(q)$ are equivalent projections in $B(H)$. All compact projections are finite rank, so it cannot be the case that the range of $\pi$ lies within the compacts. As every irreducible representation of a $C^*$-algebra with continuous trace has range within the compact operators (\cite[Thm. 6.1.11]{Pedersen}), we see that $A$ does not have continuous trace. 
\end{proof}

\begin{corollary}
If $\Lambda$ is a row-finite $k$-graph with no sources that contains a generalized cycle with entrance, then $C^*(\Lambda)$ does not have continuous trace. 
\end{corollary}

It is somewhat unsatisfactory that the question of when a higher-rank graph yields a continuous-trace $C^*$-algebra should have such a partial answer in comparison with the graph case. However this is somewhat in line with the case of other $C^*$-algebraic properties: it is difficult to decide when a $k$-graph yields an AF algebra or a purely infinite (simple) $C^*$-algebra, whereas for the graph case it is straightforward.

\bibliographystyle{plain}

\bibliography{Bibliography.bib}

\end{document}